\documentclass[12pt]{amsart}
\usepackage[hmargin=2.5cm,vmargin=3.8cm]{geometry} 
\usepackage[utf8]{inputenc}
\usepackage[english]{babel}
\usepackage{hyperref,amsmath,dsfont,amsthm,mathtools,enumerate,xcolor,mathrsfs,enumitem,amssymb,enumitem,graphicx,comment}

\usepackage[all,knot,cmtip]{xy}
\usepackage[titletoc,title]{appendix}
 \usepackage{color}
 
\usepackage{enumerate}
 \newcommand{\ok}{\overline k}
 
\newcommand{\ot}{\otimes}
\newcommand{\op}{\mathrm{op}}
\newcommand{\Z}{\mathbb{Z}} 
\newcommand{\Q}{\mathbb{Q}}

\newcommand{\C}{\mathbb{C}}
\newcommand{\CP}{\mathbb{C}P}

\renewcommand{\L}{\mathbb L}

\newcommand{\Cc}{\mathcal C^*}

\renewcommand{\a}{\mathfrak a}
\newcommand{\G}{\mathscr G}

\newcommand{\Aaut}{\mathscr Aut}

\DeclareMathOperator{\hd}{rhd}

\DeclareMathOperator{\id}{id} 
 \DeclareMathOperator{\Aut}{Aut}
 
\DeclareMathOperator{\map}{map}

\DeclareMathOperator{\aut}{hAut}
\DeclareMathOperator{\haut}{hAut}

\DeclareMathOperator{\Der}{Der}

\newtheorem{thm*}{Theorem}
\newtheorem{thm}{Theorem}[section]

\newtheorem{prop}[thm]{Proposition}

\newtheorem{lemma}[thm]{Lemma}

\theoremstyle{definition}
\newtheorem{dfn}[thm]{Definition}
\newtheorem{exmp}[thm]{Example}
\newtheorem{rmk}[thm]{Remark}
\newtheorem{conv}[thm]{Convention}

\newcommand{\Addresses}{{
  \bigskip
  \footnotesize
  
\par\nopagebreak
  \textit{E-mail address:} \texttt{basharsaleh1@gmail.com }
}}

\begin{document}
\title[Algebraic groups of homotopy classes of automorphisms]{Algebraic groups of homotopy classes of automorphisms and operadic Koszul duality} 
\author{Bashar Saleh} 
\date{}
\maketitle
\begin{abstract} Given a simply connected finite CW-complex $X$, there are several, a priori different, algebraic groups whose groups of $\Q$-points are isomorphic to the group of homotopy classes of homotopy automorphisms of the rationalization of $X$.
We will show that two of these different algebraic groups are isomorphic using the theory of operadic Koszul duality. As a by-product of the techniques we use, we deduce some isomorphisms of sets of homotopy classes of maps that might be viewed as Eckmann-Hilton dual to some well-known isomorphisms.
\end{abstract}

\section{Introduction}
Given a minimal finite type Sullivan algebra $(\Lambda V,d)$ (abbreviated $\Lambda V$ in this introduction) over $\Q$ with finite dimensional homology (i.e. $\sum_i\dim(H^i(\Lambda V))<\infty$), there are several ways of regarding the group of homotopy classes of automorphisms of $\Lambda V$, denoted ahead by $\Aut^h(\Lambda V)$, as the group of $\Q$-points of some linear algebraic group $\G$. What these  algebraic groups have in common is the following:
Even though $\Lambda V$ is not finitely generated, $\Aut^h(\Lambda V)$ might be identified with a quotient 
$$
\Aut(A)/K
$$
where $A$ is some algebraic structure that is finitely generated, and therefore gives rise to an algebraic group $\Aaut(A)$, and where $K$ happens to be the group of $\Q$-points of a normal unipotent algebraic subgroup $\mathscr K(A)\subset\Aaut(A)$. It follows that we might identify $\Aut^h(\Lambda V)$ with the group of $\Q$-points of the algebraic group $\Aaut(A)/\mathscr K(A)$. We will describe two of these constructions:
\begin{itemize}
    \item[i.] Using a certain topological argument, one can deduce that $\Aut^h(\Lambda V) \cong \Aut^h(\Lambda V^{\leq n})$ where $n$ is any natural number for which $H^{\geq n+1}(\Lambda V) =0$. Now we can regard $\Aut^h(\Lambda V^{\leq n})$ as the group of $\Q$-points of an algebraic group $$\Aaut^h(\Lambda V^{\leq n})=\Aaut(\Lambda V^{\leq n})/\mathscr K(\Lambda V^{\leq n})$$
    (this is originally due to Sullivan \cite{sullivanRHT}).
    \item[ii.] If $H^1(\Lambda V)=0$,  and $(\L W,d)$ (abbreviated $\L W$ in this introduction)  is the minimal dg Lie algebra model for the geometric realization of $\Lambda V$, it follows that $ \Aut^h(\Lambda V)=\Aut^h(\L W)$ (by some rational homotopy theory). Moreover, if $\Lambda V$ has finite dimensional homology,  it follows that $\L W$ is finitely generated (since $W=s^{-1}H^*(\Lambda V)^\vee$) and thus we might regard $\Aut^h(\L W)$ as the group of $\Q$-points of an algebraic group 
    $$\Aaut^h(\L W)=\Aaut(\L W)/\mathscr K(\L W).$$
\end{itemize}
The objectives of this paper are the following:
\subsection{Isomorphisms of algebraic groups:} In \cite{espic-saleh} the following was announced but without a full proof:
\begin{thm*}\label{main-thm-rht}
Let $X$ be a finite CW-complex with a minimal Sullivan model $\Lambda V$ and a minimal chain Lie algebra model $\L W$.
For  a sufficiently large $n$, there is an isomorphism of algebraic groups,
$$\Aaut^h(\L W)\cong \Aaut^h(\Lambda V^{\leq n}).$$
\end{thm*}
A full proof is given in this paper. This result is an essential ingredient in \cite{espic-saleh} for proving that the group of homotopy classes of relative homotopy automorphisms is finitely presented. The result is also needed in \cite{kupers-Rel-SW} in which mapping class groups of manifolds with boundary are proved to be of finite type.
We will explain why such a result is needed. The rational homotopy theory of spaces of endomorphisms of a space $X$ that fix a subspace $A\subset X$ point-wise, which we will call the space of relative endomorphisms, is more tractable through the theory of dg Lie algebra models. In both \cite{espic-saleh} and \cite{kupers-Rel-SW} the authors are interested in rational homotopy theory of relative homotopy automorphisms. Thus, they are using the Lie approach and considering algebraic groups that have natural morphisms to $\Aaut^h(\L W)$. At this point, it is the interest of the authors to invoke a Lie version of the Sullivan version of the Sullivan-Wilkerson theorem \cite[Theorem 10.3]{sullivanRHT} (see also \cite{wilkerson} for a different approach), stating that the group of components of the homotopy automorphisms is commensurable up to finite kernel with an arithmetic subgroup of the algebraic group $\Aaut^h(\Lambda V^{\leq n})$. To the best of the author’s knowledge, it is not possible to adapt Sullivan’s arguments to the Lie case without using Sullivan models. We address this problem here by proving the theorem above, making it possible to invoke the Lie version of the Sullivan version of the Sullivan-Wilkerson theorem. 

In \cite{berglund-zeman}, the authors develop models for classifying spaces of fibrations with fibers being simply connected and either homotopic to a finite CW-complex or a finite Postnikov section. In loc. cit. it is used that  the homotopy classes of homotopy automorphisms of rationalizations of such spaces are the $\Q$-points of an algebraic group; $\Aaut^h(\L W)$ in the finite CW-complex case, and $\Aaut^h(\Lambda V)$ in the finite Postnikov section case. For an elliptic space, $\Aaut^h(\L W)$ and $\Aaut^h(\Lambda V)$ are a priori different, but by  Theorem \ref{main-thm-rht} of this paper we know that they are isomorphic and hence the two different ways of obtaining algebraic models for elliptic spaces are essentially equivalent.

\subsection{Pure algebraic proofs} Even though the isomorphism $\Aut^h(\Lambda V) \cong \Aut^h(\Lambda V^{\leq n})$ might be regarded as a pure algebraic statement, the proofs in the literature (known to the author) are based to some extent on topological arguments, justified through the following sequence of arguments (see e.g. \cite{BlockLazarev} or \cite{CostoyaViruel}): The group of homotopy classes of the automorphisms of $\Lambda V$, $\Aut^h(\Lambda V)$, is isomorphic to the group of homotopy classes of homotopy automorphisms of the geometric realization $\langle \Lambda V\rangle$, denoted by $\pi_0(\aut\langle \Lambda V\rangle)$. Using obstruction theory for lifts of maps of topological spaces, it is possible to prove that if $H^*(X)$ is concentrated in degrees $\leq n$ then $\pi_0(\aut(X))=\pi_0(\aut(X(n)))$ where $X(n)$ denotes the $n$'th Postnikov stage of $X$. Through the equivalence $\langle\Lambda V\rangle (n)\simeq \langle \Lambda V^{\leq n}\rangle$, and the arguments above, one deduces that $\Aut^h(\Lambda V) = \Aut^h(\Lambda V^{\leq n})$ whenever $H^*(\Lambda V)$ is concentrated in degrees $\leq n$.

A morphism of algebraic groups over a field $k$ of characteristic zero is an isomorphism if and only if the induced morphism on the $\overline k$-points is an isomorphism (where $\overline k$ denotes any algebraic closed field containing $k$). Hence we will be interested in the groups $\Aut^h(\Lambda V \ot_\Q \C)$ and $\Aut^h(\Lambda V^{\leq n}\ot_\Q \C)$, rather than $\Aut^h(\Lambda V)$ and $\Aut^h(\Lambda V^{\leq n})$. A natural question that arises is thus whether we still might identify $\Aut^h(\Lambda V \ot_\Q \C)$ with $\Aut^h(\Lambda V^{\leq n}\ot_\Q \C)$. The topological arguments mentioned above are not applicable for other fields than $\Q$ (see \cite[Remark 3.11]{quillenRHT}), so a completely pure algebraic proof is needed. Moreover, a pure algebraic proof enables us to prove analogous results for general operad algebras. The methods are standard but written out in Section \ref{sec:alg} since the author could not find written proofs in the literature. Essentially, we consider an algebraic counterpart to parts of the established topological obstruction theory.

\subsection{Eckmann-Hilton duals to some classical isomorphisms.} The general algebraic approach is applicable to chain Lie algebra models for spaces. As a result of this, we deduce some Eckmann-Hilton duals to some well-known facts about homotopy classes of maps between finite CW-complexes and the homotopy classes of maps between their Postnikov stages.

For this we need to consider a rational Eckmann-Hilton dual notion of Postnikov stages, which we call rational homology decomposition stages, where the $n$'th rational homology decomposition stage of $X$ is denoted by $\hd(X,n)$. We also need to restrict to connected spaces whose rational homology are sparse in the following sense: For any $k\in \Z$, at least one of the reduced homology groups $\tilde H^k(X;\Q)$ and $\tilde H^{k+1}(X;\Q)$ is trivial.

\begin{thm*}
 Let $X$ be either a simply connected space or a nilpotent space of finite $\Q$-type such that $\pi_*(X)\ot\Q$ is concentrated in degrees $\leq n-1$. 
    \begin{itemize}
        \item[(a)] For any simply connected space $Y$ there are isomorphisms of sets 
            $$[Y_\Q,X_\Q]_*\cong [\hd(Y,n),X_\Q]_* \cong [\hd(Y,n+1),X_\Q]_*\cong\dots.$$
        \item[(b)] If $X$ is simply connected and rational homological sparse, then  there are isomorphisms of sets and groups, respectively,
    $$
    [X_\Q,X_\Q]_*\cong[\hd(X,n),\hd(X,n)]_*\cong[\hd(X,n+1),\hd(X,n+1)]_*\cong\dots 
    $$
    and
    $$\pi_0(\aut_*(X_\Q)) \cong\pi_0(\aut_*(\hd(X,n))\cong\pi_0(\aut_*(\hd(X,n+1))\cong\dots.$$
    \end{itemize}
\end{thm*}

\subsection{Generalization through operadic Koszul duality}
When $H^1(\Lambda V)=0$, the isomorphism $\Aut^h(\Lambda V)=\Aut^h(\L W)$ can be viewed as a manifestation of the Koszul duality between the commutative and the Lie operad. We take advantage of our algebraic proofs to prove analogous results for algebras over Koszul dual operads.
\\
\begin{center}\textbf{Acknowledgements}\end{center}
I would like to thank Alexander Kupers for informing Hadrien and me about the validity of the arguments in \cite{espic-saleh} relies on what now is the main theorem of this paper. I also thank him for answering many of my questions.

Thanks to Dan Petersen, Alexander Berglund and Nils Prigge for answering my questions and discussing  some of the material covered in this paper. Thanks to the anonymous referee for helpful comments. 

The author was supported by the Knut and Alice Wallenberg Foundation through grant no. 2022.0223.

\section{Extensions of maps and homotopies of algebras}\label{sec:alg}
In this section, we study the properties of homotopies of algebra maps and derive some isomorphisms of sets of homotopy classes of algebra morphisms.

We will use the language of operads, but the reader interested only in rational homotopy theory, might replace the term `cochain $P$-algebra' by a commutative chain algebra (cdga), and Sullivan $P$-algebras will be the ordinary Sullivan algebras. The reader migtht also replace the term `chain $P$-algebra' by a chain Lie algebra (dg Lie algebra), and 1-connected Sullivan chain $P$-algebras will be quasi-free 1-connected chain Lie algebra.

\subsection{Conventions and preliminaries}
\begin{conv}
We fix a field $k$ of characteristic zero and an  operad $P$ in the category of vector spaces over a field $k$. This operad can be viewed as an operad in the category of (co)chain complexes concentrated in (co)homological degree zero. We assume that $P$ is connected ($P(1)= k$) and that it is either unitary ($P(0)=k$) or reduced ($P(0)=0$). \end{conv}

In this paper, we will work with both \textit{chain} $P$-algebras (where the differential decreases the degree) and with \textit{cochain} $P$-algebras (where the differential increases the degree). The categories of chain and cochain algebras are equivalent (through the identification of cohomological degree $i$ with homological degree $-i$), but the notions of connectivity differs in an asymmetric way. For that reason, we will not work with a convention that fixes the sign of the degree of the differential, but rather work with both conventions in parallel.

The (co)homological degree part of a (co)homological graded object $C$ is denoted by $C^n$.

\subsubsection{(Co)Chain homotopy}
Given a (co)chain $P$-algebra $A$ and a (co)chain commutative algebra $\Omega$, the tensor product $A\ot_k \Omega$ is again a (co)chain $P$-algebra. Let $\Lambda(t,dt)$ be the unital free (co)chain commutative algebra generated by $t$ and $dt$, where the (co)homological degree of $t$ is 0 (and thus $dt$ is of cohomological degree $1$ in a cohomological convention and is of homological degree $-1$ in the homological convention).

\begin{dfn}
Two maps of (co)chain $P$-algebras $f,g\colon A\to B$ are said to be \textit{homotopic} if there exists a (co)chain P-algebra map
$$h\colon A\to B\ot_k\Lambda(t,dt)$$
such that $ev_0\circ h = f$ and $ev_1\circ h = g$. The map $h$ is called a \textit{homotopy} from $f$ to $g$.
\end{dfn}

In particular, $h$ is a sum of the form
$$
h = \sum_{i\geq 0}(\alpha_i t^i + \beta_i t^i dt)
$$
where $\alpha_i$ and $\beta_i$ are linear maps $A\to B$. If $h$ is a homotopy from $f$ to $g$ then
\begin{equation}\label{eqn:sum-ai}
\alpha_0 = f \qquad \text{and}\qquad \sum_{i\geq 0}\alpha_i = g.
\end{equation}

\subsubsection{Cochain connectivity}
A cochain $P$-algebra $A$ is said to be \textit{connected} if it is concentrated in non-negative degrees and if $A^0=P(0)$. We say that $A$ is \textit{1-connected} if $A$ is connected and  $A^1=0$.

\subsubsection{Chain connectivity}
We say that a chain  $P$-algebra $A$ is \textit{connected} if it is concentrated in non-negative homological degrees. We say that $A$ is \textit{1-connected} if it is connected and $A^0=P(0)$.

\subsubsection{1-connected minimal models}\label{subsubsec:mini}
A (co)chain $P$-algebra $(A,d)$ is called quasi-free if the underlying graded $P$-algebra $A$ is a free graded $P$-algebra, $A =P(V)$, generated by a graded vector space $V= \bigoplus_{i\in\Z} V^i$.
We say that a 1-connected quasi-free (co)chain $P$-algebra $(P(V),d)$ is \textit{minimal} if $d(V)\subset P^{(n\geq 2)}(V) = \bigoplus_{i\geq 2} P(i)\ot_{k[\Sigma_i]} V^{\ot i}$. By degree reasons, it follows that 
    $d(V^n)\subset P(V^{\leq n-1})$. 
\begin{rmk}
    This notion of minimality is only suitable for 1-connected (co)chain algebras over operads concentrated in degree 0. For a general notion of minimality, we refer to \cite{CiriciRoig}, where a quasi-free algebra is minimal if it is given by a direct limit of minimal KS-extensions. However, for a 1-connected quasi-free (co)chain $P$-algebra $(P(V),d)$, where $P$ is concentrated in degree 0, and where $(P(V),d)$ is minimal in our sense, we have that the sequence of inclusions
    $$
    P(0)\to P(V^1) \to P(V^{\leq 2}) \to \dots
    $$
    is a sequence of minimal KS-extensions in the sense of Cirici-Roig \cite{CiriciRoig}, and thus minimal in their sense as well.
\end{rmk}

By the theory of \cite{CiriciRoig} (which applies for a far wider variety of operads and algebras) we know the following:
\begin{itemize}
    \item[(a)] Simply connected (co)chain $P$-algebra admits minimal models, unique up to isomorphisms (\cite[Theorem 4.6, Theorem 5.3]{CiriciRoig}).
    \item[(b)] Quasi-isomorphisms between minimal (co)chain $P$-algebras are isomorphisms (\cite[Theorem 5.2]{CiriciRoig}).
\end{itemize}

\begin{conv}\label{conv:lowerindex}
    Given a 1-connected minimal (co)chain $P$-algebra $A=(P(V),d)$, we let $A_n\subset A$ denote the (co)chain subalgebra $(P(V^{\leq n}),d|_{P(V^{\leq n})})$. 
\end{conv}

\subsection{Homotopies and extensions} In this subsection, we consider a sort of obstruction theory for extension of algebra maps. The algebraic theory is elementary, but not covered in the literature known to the author (not even for Sullivan algebras).

We believe that this type of obstruction theory is not easily found in the standard literature for the following reason: Over $\Q$, the obstruction theory for extensions of maps (up to homotopy) from a Sullivan commutative cochain algebra, as developed here, can be derived in a different way by first taking realizations and then applying standard topological obstruction theory. This topological obstruction theory was established prior to the introduction of rational homotopy theory, and hence in those particular cases, there is no need for an algebraic obstruction theory. However, in this paper, we require an obstruction theory for algebras over algebraically closed fields containing the rationals, which means we cannot rely on the obstruction theory for spaces. 

We will only consider 1-connected (co)chain algebras, even though most of the statements in this section are true for connected cochain algebras that are minimal in the sense of \cite{CiriciRoig}. However, we have chosen to not treat these cases due to applications on mind and since these cases would make the proofs more complicated. 

It is also possible to generalize some of the results to certain operads that has a non-trivial (co)homological grading.

\begin{lemma}\label{lemma:alpha-depends-on-beta}
Let $h =  \sum_{i\geq 0}(\alpha_i t^i + \beta_i t^i dt)$ be a homotopy $h\colon A\to B\ot_k\Lambda(t,dt)$ of (co)chain $P$-algebras. It follows that $\alpha_i$ commutes with the differential and that
\begin{equation}\label{eqn:alpha=[d,beta]}
-\alpha_{i+1} = \frac{d\beta_i+\beta_i d}{i+1},\quad i\geq0.
\end{equation}
In particular, a homotopy is completely determined by $\alpha_0$ and $\{\beta_i\}_{i\geq 0}$.
\end{lemma}

\begin{proof}
For an arbitrary element $x\in A$, we have that  
$$
dh(x) = \sum_{i\geq 0}d\alpha_i(x)t^i + \sum_{i\geq 0}(-1)^{|x|}((i+1)\alpha_{i+1}(x) +d\beta_i(x)) t^idt
$$
and
$$
hd(x) = \sum_{i\geq 0}\alpha_i(dx)t^i+\sum_{i\geq 0}(-1)^{|x|+1}\beta_i(dx)t^idt 
$$
Since $h$ commutes with the differential, it follows by above expansions that $d\alpha_0=\alpha_0d$ and that $(i+1)\alpha_{i+1} = -(\beta_id + d\beta_i)$.
\end{proof}

\begin{lemma}\label{lemma:extensions-of-htpy}
Let $A =P(V)$ be a 1-connected minimal quasi-free (co)chain $P$-algebra  and let $B$ be a (co)chain $P$-algebra whose (co)homology is concentrated in (co)homological degrees $\leq n$.
Given a homotopy
$$
h_n\colon A_n= P(V^{\leq n})\to B\ot_k\Lambda(t,dt)
$$
from a map $f_n$ to a map $g_n$ and given extensions $f_{n+1}$ and $g_{n+1}$  of $f_n$ and $g_n$, respectively, to $A_{n+1} = P(V^{\leq n+1})$, there is an extension $h_{n+1}\colon A_{n+1}\to B \ot_k \Lambda(t,dt)$  of $h_n$ that defines a homotopy  from $f_{n+1}$ to $g_{n+1}$.
\end{lemma}
\begin{proof}
 We need to determine the restrictions of $\alpha_i$ and $\beta_i$, $i\geq 0$,  to  $V^{n+1}$. By Lemma \ref{lemma:alpha-depends-on-beta}, it is enough to determine $\alpha_0$ and $\beta_i$, $i\geq 0$ restricted to $V^{n+1}$. We start by fixing a basis for $V^{n+1}$. For a basis vector $v$ we do the following:

    We set $\alpha_0(v) = f_{n+1}(v)$ (there is no other choice by \eqref{eqn:sum-ai}).
    Since $v\in V^{n+1}$, we have minimality that  $dv\in A_n$ (see § \ref{subsubsec:mini}), and thus $\beta_i(dv)$ is already determined. We observe that
    \begin{align*}
d\left(
 g_{n+1}(v)-f_{n+1}(v)+\sum_{i\geq 0}\frac{\beta_id(v)}{i+1}
 \right) &= g_n(dv)-f_n(dv)+\sum_{i\geq 0}\frac{d\beta_id(v)}{i+1}\\
 &=g_n(dv)-f_n(dv)+\sum_{i\geq 0}\frac{d\beta_id(v)+\beta_i d^2(v)}{i+1}\\
 &\overset{\eqref{eqn:alpha=[d,beta]}}{=} g_n(dv)-f_n(dv)-\sum_{i\geq 0}\alpha_{i+1}(dv)\\
 &\overset{\eqref{eqn:sum-ai}}= g_n(dv)-f_n(dv)- (g_n(dv)-f_n(dv)) = 0.
\end{align*}

 Hence $ g_{n+1}(v)-f_{n+1}(v)+\sum_{i\geq 0}\frac{\beta_id(v)}{i+1}$ is a cycle of degree $n+1$ and thus also a boundary (since $H^{n+1}(B)=0$). Pick some $z\in B$ such that $dz =  g_{n+1}(v)-f_{n+1}(v)+\sum_{i\geq 0}\frac{\beta_id(v)}{i+1}$ and let $\beta_0(v)=z$ and $\beta_i(v)=0$ for $i\geq 0$.

Doing this for every basis vector, we get extensions of  $\alpha_0$ and $\{\beta_i\}_{i\geq 0}$ to $V^{\leq n+1}$, which defines a map of graded algebras $h_{n+1}\colon A_{n+1}\to B\ot_k\Lambda(t,dt)$.
Straightforward expansions of $h(dv)$ and $dh(v)$ for $v\in V^{n+1}$, shows that $h$ commutes with the differential, and hence is a morphism (co)chain $P$-algebras thus defining a homotopy.

 We have then that $ev_0\circ h_{n+1}(v) = \alpha_0(v)= f_{n+1}(v)$ and that 
 \begin{align*}
     ev_1\circ h_{n+1}(v)&= \sum_{i\geq 0}\alpha_i(v) 
     \\& = f_{n+1}(v) - \sum_{i\geq 0}\frac{d\beta_i(v)+\beta_i(dv)}{i+1}\\
     &= f_{n+1}(v)-d\beta_0(v) -\sum_{i\geq 0}\frac{\beta_id(v)}{i+1}\\
     &= f_{n+1}(v) +\left(g_{n+1}(v)-f_{n+1}(v)+\sum_{i\geq 0}\frac{\beta_id(v)}{i+1} \right)-\sum_{i\geq 0}\frac{\beta_id(v)}{i+1}\\
     &=g_{n+1}(v)
 \end{align*}

This proves that $h_{n+1}$ is the desired extended homotopy.\end{proof}

\begin{lemma}\label{lemma:inj}
Under the assumptions of Lemma \ref{lemma:extensions-of-htpy}, the restriction map $\map(A_{n+1},B)\to \map(A_n,B)$
induces an injective map on the set of homotopy classes of maps $[A_{n+1},B]\to [A_n,B]$.
\end{lemma}
\begin{proof}
Given a homotopy $h_{n+1}\colon A_{n+1}\to B\ot_k \Lambda(t,dt)$ from $f$ to $g$, we have that its restriction to $A_n$ defines a homotopy from the restriction of $f$ to $A_n$ to the restriction of $g$ to $A_n$. Hence, there is a well-defined map $[A_{n+1},B]\to [A_n,B]$, $[f] \mapsto [f|_{A_n}]$. It is injective since homotopies can be extended by the previous lemma, Lemma \ref{lemma:extensions-of-htpy}.
\end{proof}

We aim to determine when the restriction induced map $[A_{n+1},B]\to [A_n,B]$ is an isomorphism. We deduced injectivity whenever the (co)homology of $B$ is concentrated in degrees less or equal to $n$ by showing that extensions to homotopic maps are homotopic. In order to prove surjectivity, we prove that extensions always exist.

\begin{lemma}\label{lemma:surj}
\begin{itemize}
    \item[(a)] (Cochain version) Let $A =(P(V),d_A)$ be a 1-connected minimal quasi-free cochain $P$-algebra  and let $B$ be a cochain $P$-algebra whose cohomology is concentrated in cohomological degrees $\leq n$.
    Given a map $f_{n-1}\colon A_{n-1}\to B$, there is an extension $f_n\colon A_n\to B$.
    
    \item[(b)] (Chain version) Let $A =(P(V),d_A)$ be a 1-connected minimal quasi-free chain $P$-algebra  and let $B$ be a chain $P$-algebra whose homology is concentrated in homological degrees $\leq n$. Given a map $f_{n+1}\colon A_{n+1}\to B$, there is an extension $f_{n+2}\colon A_{n+2}\to B$.
\end{itemize}
\end{lemma}
\begin{proof} (a)
$A_n$ is by definition given as a pushout
$$
\xymatrix{
(P(sV^{n}),d(sv)=0)\ar[r]^-{sv\mapsto d_A(v)}\ar@{^(->}[d] & A_{n-1}\ar[d] \\
(P(sV^n\oplus V^n),d(v)=sv)\ar[r] & A_n   
},
$$
where $sV^n$ is the suspension of $V^n$, concentrated in degree $n+1$.

The composition $F\colon P(sV^{n})\to  A_{n-1}\xrightarrow{f_{n-1}} B$ lands in $Z^{\geq n+1}(B)$ and hence in $B^{\geq n+1}(B)$ (since $H^{\geq n+1}(B)=0$). We fix a basis for $V^{n+1}$ and for every basis vector $v$, we fix an element $b_v$ where $d_B(b_v)= F(v)$. 
Thus, there is a commutative diagram 
$$
\xymatrix{
(P(sV^{n}),d(sv)=0)\ar[rr]^-{sv\mapsto d_A(v)}\ar@{^(->}[d] && A_{n-1}\ar[d]^{f_{n-1}} \\
(P(sV^n\oplus V^n),d(v)=sv)\ar[rr]_-{\begin{smallmatrix}
sv\mapsto F(v)\\
v\mapsto b_v
\end{smallmatrix}
} && B  
}
$$
By the universal property of pushouts there is a map $f_n\colon A_n\to B$, extending $f_{n-1}$.

(b) The proof follows the same route except that we need to replace $sV^n$ by $s^{-1}V^{n+2}$.
\end{proof}

\begin{thm}\label{thm:iso-of-htpy-classes}
\begin{itemize}
    \item[(a)] (Cochain version)  Let $A =P(V)$ be a 1-connected minimal quasi-free cochain $P$-algebra  and let $B$ be a cochain $P$-algebra whose cohomology is concentrated in cohomological degrees $\leq n$. Then 
    $$[A_n,B]\cong[A_{n+1},B]\cong \dots\cong [A,B]$$
    \item[(b)] (Chain version) Let $A =P(V)$ be a 1-connected minimal quasi-free chain $P$-algebra  and let $B$ be a chain $P$-algebra whose homology is concentrated in homological degrees $\leq n$.  Then 
    $$[A_{n+1},B]\cong[A_{n+2},B]\cong \dots\cong [A,B]$$
\end{itemize}
\end{thm}

\begin{proof}
    (a) The (well-defined) map $[A_{n+k},B]\to [A_{n+k-1},B]$, $[f]\mapsto [f|_{A_{n+k-1}}]$ is injective for $k\geq 1$ by Lemma \ref{lemma:inj}. It is surjective for $k\geq 0$ by Lemma \ref{lemma:surj}.

    (b) This is proved similarly.
\end{proof}

\subsection{Extensions of endomorphisms and automorphisms}
Given a quasi-free (co)chain $P$-algebra  $A=(P(V),d)$, we might identify $\map(A_n,A)$ with $\map(A_n,A_n)$, since any map $A_n\to A$ has an image in $A_n\subset A$. However, we can only identify $[A_n,A]$ with $[A_n,A_n]$ if $A$ is a cochain algebra (and not in general if $A$ is a chain algebra). We explain why: In the cohomological setting, $dt$ is of cohomological degree 1, hence the image of any homotopy $h\colon A_n\to A\ot_k\Lambda(t,dt)$ is a subspace of $A_n\ot_k\Lambda(t,dt)$. Consequently, $h$ might be viewed as a homotopy of endomorphisms of $A_n$. However, in the homological setting, $dt$ is of homological degree $-1$, and thus the image of $h$ is a subspace of $A_{n+1}\ot_k \Lambda(t,dt)$. However, for a special class of chain algebras, such identifications are still possible.

\begin{dfn}
    We say that a quasi-free chain $P$-algebra $P(V)$ is \textit{sparsely generated} if for any consecutive pair of integers $k,k+1$, at least one of the vector spaces, $V^k$ and $V^{k+1}$, is trivial.
\end{dfn}

\begin{rmk}
    A 1-connected sparsely generated chain $P$-algebra is automatically minimal.
\end{rmk}

By degree reasons, we have the following.
\begin{lemma}\label{lemma:aa}
    Let $A=P(V)$ be either a 1-connected minimal cochain $P$-algebra or a 1-connected sparsely generated chain $P$-algebra.
    Then there is an isomorphism of sets $[A_n,A]\cong [A_n,A_n]$.
\end{lemma}

\begin{thm}\label{thm:iso-of-endo}
    Let $A=P(V)$ be either a 1-connected minimal cochain $P$-algebra or a 1-connected sparsely generated chain $P$-algebra. Assume that $H^*(P(V))$ is concentrated in degrees $\leq n$.
    Then there are isomorphism of sets
    $$ [A_k,A_k] \cong [A_{k+1},A_{k+1}]\cong\dots\cong[A,A],$$
    where $k=n$ in the cochain algebra case and $k=n+1$ in the sparsely generated chain algebra case. The isomorphisms are realized by the well-defined map $[A,A]\to [A_{k+s},A_{k+s}]$, $[\varphi]\mapsto [\varphi|_{A_{k+s}}]$.\end{thm}
\begin{proof}
    This follows immediately from Theorem \ref{thm:iso-of-htpy-classes} and Lemma \ref{lemma:aa}.
\end{proof}

When the map $[A,A] \to [A_k,A_k]$ induced by restriction is defined and is an isomorphism, it follows that a homotopy class of an  automorphism of $A$ is mapped to a homotopy class of an automorphism of $A_k$. In these cases we have a well-defined group homomorphism 
$$
\Aut^h(A)\to \Aut^h(A_k).
$$
We will prove that this group homomorphism is an isomorphism.

\begin{prop}\label{prop:ext-of-iso-is-iso}
    Assume that $f_n\colon A_n\to A_n$ is an isomorphism and let $f_{n+1}\colon A_{n+1}\to A_{n+1}$  be any extension of $f_n$. Then $f_{n+1}$ is also an isomorphism.
\end{prop}

\begin{proof}
Let $g_n\colon A_n\to A_n$ be the inverse of $f_n$. Let $f_{n+1},g_{n+1}\colon A_{n+1}\to A_{n+1}$ be extensions of $f_n$ and $g_n$ respectively (they exist by Lemma \ref{lemma:surj}). 

We have that $f_{n+1}\circ g_{n+1}$ and $g_{n+1}\circ f_{n+1}$ are extensions of $f_{n}\circ g_{n}=\id_{A_n}$ and $g_{n}\circ f_{n}=\id_{A_n}$.
By Lemma \ref{lemma:extensions-of-htpy}, it follows that $f_{n+1}\circ g_{n+1}$ and $g_{n+1}\circ f_{n+1}$ are homotopic to $\id_{A_{n+1}}$, and hence $f_{n+1}\circ g_{n+1}$ and $g_{n+1}\circ f_{n+1}$ are quasi-isomorphisms. By minimality of $A_{n+1}$, it follows that $f_{n+1}\circ g_{n+1}$ and $g_{n+1}\circ f_{n+1}$ are isomorphisms (and in particular that $f_{n+1}$ is an isomorphism).
\end{proof}

\begin{thm}\label{thm:iso-of-auto}
    Let $A=P(V)$ be either a 1-connected minimal cochain $P$-algebra or a 1-connected sparsely generated chain $P$-algebra. Assume that $H^*(P(V))$ is concentrated in degrees $\leq n$.
    Then there are isomorphism of groups
    $$ \Aut^h(A_k) = \Aut^h(A_{k+1})=\dots=\Aut^h(A),$$
    where $k=n$ in the cochain algebra case and $k=n+1$ in the sparsely generated chain algebra case. The isomorphisms are realized by the well-defined map $\Aut^h(A)\to \Aut^h(A_{k+2})$, $[\varphi]\mapsto [\varphi|_{A_{k+s}}]$.
\end{thm}
\begin{proof}
The restriction induced isomorphism $[A_{k+s+1},A_{k+s+1}]\xrightarrow\cong [A_{k+s},A_{k+s}]$ of Theorem \ref{thm:iso-of-endo}, takes a homotopy class of an  automorphism of $A_{k+s+1}$ to a homotopy class of an automorphism of $A_{k+s+1}$. By Proposition \ref{prop:ext-of-iso-is-iso}, the preimage of a homotopy class of an automorphism of $A_{k+s}$ is a homotopy class of an automorphism of $A_{k+s+1}$. This gives the desired isomorphism.
\end{proof}

\subsection{Some Eckmann-Hilton duals to some classical isomorphisms}
Given an\break $n$-dimensional CW-complex $X$, we have the following isomorphisms of sets
$$
[X,Y] \cong [X,Y(n)],\qquad [X,X]\cong[X(n),X(n)],\qquad\pi_0(\haut(X))\cong \pi_0(\haut(X(n)))
$$
(the last one being an isomorphism of groups).
We could view these results as topological counterparts to the cochain versions of the isomorphisms in Theorem \ref{thm:iso-of-htpy-classes}, Theorem \ref{thm:iso-of-endo} and Theorem \ref{thm:iso-of-auto},  respectively. A natural question is whether there are topological counterparts to the chain versions of these theorems. Through the theory of chain Lie algebra models for spaces we will be able to derive some kind of Eckmann-Hilton dual isomorphisms, valid for rational spaces.

We recall some elementary rational homotopy theory. The rational homotopy type of a simply connected space or a nilpotent space of finite type is encoded by some chain Lie algebra model, unique up to quasi-isomorphism. If  $L$ is a chain Lie algebra model for $X$, then we have that $H^*(L)= \pi_{*+1}(X)\ot_{\Z}\Q$, where $\pi_1(X)\ot \Q$ is to be interpreted as the Malcev completion of $\pi_1(X)$. Note that we will use an upper index on \( L \) to indicate the homological degree for the chain Lie algebra model, which is uncommon. Therefore, the $n$'th homology group is denoted by \( H^n(L) \) rather than \( H_n(L) \). The reason for this is that the upper index was  used to refer to both homological and cohomological degrees earlier in this section. The lower index is reserved for another purpose (see Convention \ref{conv:lowerindex}.

Moreover, if $L$ and $L'$ are quasi-free chain Lie algebras for  $X$ and $X'$ respectively, then the set of homotopy classes of maps $[L,L']$ is isomorphic to the set of base-point preserving homotopy classes of base-point preserving maps from $X_\Q$ to $X'_\Q$, denoted by $[X_\Q,X_\Q']_*$.

If $X$ is simply connected, then it admits a minimal chain Lie algebra model, $(\L(V),d)$, unique up to isomorphism.

We also recall that there are several geometric realization functors of chain Lie algebras $\mathfrak g\mapsto |\mathfrak g|$, all being equivalent up to homotopy (see e.g. \cite{ffm22} and \cite{berglundexponential}), and satisfying the following: If $\mathfrak g$ is a chain Lie algebra model for $X$, then $|\mathfrak g|$ is weakly equivalent to the rationalization $X_\Q$ of $X$.

\begin{dfn}
A simply connected space $X$, has a unique minimal chain Lie algebra model $L=(\L(V),d)$. We define $\hd(X,n)$ to be the geometric realization $|L_{n-1}|$ of $L_{n-1} = (\L(V^{\leq n-1}),d)$.
\end{dfn}

We should think of $\hd$ as some kind of rational homology decomposition $n$-stage since
$$H_i(\hd(X,n))=\left\{\begin{matrix} H_i(X;\Q)& i\leq n\\0&i>n.
\end{matrix}\right.$$

We explain this further; if $(\L(V),d)$ is a minimal chain Lie algebra model for $X$, then $V = s^{-1}\tilde H_*(X;\Q)$. In particular, $\L(V)$ is sparsely generated if and only if $\tilde H_*(X;\Q)$ is concentrated in sparse degrees in the following sense:

\begin{dfn}
    We say that a space $X$ is \textit{rational homological sparse} if its rational homology satisfies the following: For any consecutive pair of integers $k, k+1$, at least one of the reduced homology groups $\tilde H_k(X;\Q)$ and $\tilde H_{k+1}(X;\Q)$, is trivial.
\end{dfn}

\begin{thm}\label{thm:eckmann-hilton}
    Let $X$ be either a simply connected space or a nilpotent space of finite $\Q$-type such that $\pi_*(X)\ot\Q$ is concentrated in degrees $\leq n-1$. 
    \begin{itemize}
        \item[(a)] For any simply connected space $Y$ there are isomorphisms of sets 
            $$[\hd(Y,n),X_\Q]_* \cong [\hd(Y,n+1),X_\Q]_*\cong\dots \cong [Y_\Q,X_\Q]_*.$$
        \item[(b)] If $X$ is simply connected and rational homological sparse, then  there are isomorphisms of sets and groups, respectively,
    $$
    [\hd(X,n),\hd(X,n)]_*\cong[\hd(X,n+1),\hd(X,n+1)]_*\cong\dots \cong [X_\Q,X_\Q]_*
    $$
    and
    $$\pi_0(\aut_*(\hd(X,n))\cong\pi_0(\aut_*(\hd(X,n+1))\cong\dots \cong \pi_0(\aut_*(X_\Q))$$
    \end{itemize}
\end{thm}

\begin{proof} (a) Let $L$ be a chain Lie algebra model for $X$. Since $\pi_*(X)\ot\Q$ is concentrated in degrees $\leq n-1$ it follows that $H^*(L)$ is concentrated in degrees $\leq n-2$. Since $Y$ is simply connected, it has a minimal chain Lie algebra model $\L(W)$. By the chain version (the \textit{(b)}-part) of Theorem \ref{thm:iso-of-htpy-classes} we have that 
$$
[\L(W^{\leq n-1}),L]\cong [\L(W^{\leq n}),L] \cong \dots\cong [\L(W),L].
$$
The \textit{(a)}-part follows since $[\L(W^{\leq k}),L] = [\hd(Y,k+1),X_\Q]_*$.

(b) Let $L=\L(V)$ be a minimal chain Lie algebra model for $X$. Similar to above, the assertions follow from the chain versions of Theorem \ref{thm:iso-of-endo} and Theorem \ref{thm:iso-of-auto}.
\end{proof}

\begin{exmp}
We have that $\CP^\infty_\Q = K(\Q,2)$, and its reduced homology is concentrated in even degrees (and thus in sparse degrees). Moreover, its minimal chain Lie algebra model is given by the quasi-free chain Lie algebra $(\L(x_1,x_2,\dots),d)$ where $|x_m|=2m-1$ and where $dx_m =\frac{1}{2}\sum_{i+j=m}[x_i,x_j]$. Moreover, $(\L(x_1,\dots,x_m),d|_{\L(x_1,\dots,x_m)})$ is a minimal model for $\CP^m$. In particular, $\CP^m_\Q \simeq \hd(\CP^\infty,2m)\cong \hd(\CP^\infty,2m+1)$ It follows now by Theorem \ref{thm:eckmann-hilton} that
$$[\CP^2_\Q,\CP^2_\Q]_*\cong [\CP^3_\Q,\CP^3_\Q]_*\cong \dots\cong [\CP^\infty_\Q,\CP^\infty_\Q]_*\cong \Q$$
and
$$\pi_0(\haut_*(\CP^2_\Q))\cong\pi_0(\haut_*(\CP^3_\Q))\cong\dots\cong\pi_0(\haut_*(\CP^\infty_\Q))\cong \Q^\times.$$
By analogous reasoning  there are analog isomorphisms for the rationalizations of quaternionic projective spaces.
\end{exmp}

\section{Algebraic groups and operadic Koszul duality}

\subsection{Preliminaries on linear algebraic groups}
An algebraic group $\mathscr G$ is a group object in the category of varieties. We say that $\mathscr G$ is a linear algebraic group if it is a subgroup object of $\mathscr GL_n$ for some $n\geq 1$, the algebraic group of invertible $n\times n$-matrices. We start by introducing some notation.

\begin{dfn}
Given a (co)chain $P$-algebra $A$, we have that all structure maps $\mu_k\colon A^{\ot k}\to A$ (including the differential on A) induce maps $\mu_k^{\leq n}\colon (A^{\leq n})^{\ot k}\to A^{\leq n}$ where 
$$
\mu_k^{\leq n}(x_1\ot\cdots\ot x_n)=\left\{\begin{array}{ll}
     \mu_k(x_1\ot\cdots\ot x_n)&\quad \text{if $\mu_k(x_1\ot\cdots\ot x_n)\in A^{\leq n}$,}  \\
     0&\quad  \text{else.}
\end{array}\right.
$$
We call the maps $\{\mu_k^{\leq n}\}_k$ for the \textit{$n$-restricted structure maps}. 
\end{dfn}

\begin{dfn} Given a connected (co)chain $P$-algebra $A$ of finite type, let $\mathscr G_P^n(A)$ denote the linear algebraic group given by all automorphisms of $A^{\leq n}$ that commutes with the  $n$-restricted structure maps. This is an algebraic group due to linearity of the structure maps.
\end{dfn}

 For a connected finitely generated finite type quasi-free (co)chain $P$-algebra $A=P(V^{\leq n}) $ there are isomorphisms of groups 
\begin{equation}\label{eqn:r}
\Aut(A\ot_k R) \cong  \mathscr G_P^{n+\epsilon}(A)(R) 
\end{equation}
for any ring $R$ over $k$ and where $\epsilon = 0$ if $A$ is a chain algebra and $\epsilon=1$ if $A$ is a cochain algebra. The reason for the variation of $\epsilon$ is due to the degree  of the differential; if  $d$ raises the degree we need to keep track of the degree $n+1$ decomposable elements that are boundaries of generators of degree $n$. We let $\Aaut(A)$ denote the algebraic group $\G^{n+\epsilon}_P(A)$.

\begin{rmk}\label{rmk:aut-as-gn+t}
Clearly, if the isomorphism \eqref{eqn:r} holds, the algebraic group $\mathscr G^m_P(A)$ is isomorphic to $\Aaut(A)$ for any $m\geq n+\epsilon$. In particular
$$\Aaut(A)\cong \G^{n+1}_P(A)\cong\G^{n+2}_P(A)\cong\dots.$$    
\end{rmk}

We will state some facts about linear algebraic groups that will be needed later.

\begin{lemma}[\text{\cite[Proof of Theorem 9.5]{wilkerson}, not affected by the erratum \cite{wilkersonerrata}}]\label{wilkerson}
Given a linear algebraic group $\G$ over a perfect field $k$ (e.g. a field of characteristic zero) and a unipotent linear algebraic subgroup $\mathscr U\subset \G$, we have that $(\G/\mathscr U)(k) = \G(k)/\mathscr U(k)$.
\end{lemma}
\begin{rmk}
    Wilkerson \cite{wilkerson} treats the case $k=\Q$, but the arguments work equally well for perfect fields since $H^1(\mathrm{Gal}(\overline k/k), \mathscr U)=0$ whenever $k$ is perfect and $\mathscr U$ is unipotent (see \cite[Proposition 6, Chapter III]{serre-galois}).
\end{rmk}

\begin{lemma}[\text{\cite[Theorem 3.4]{BlockLazarev}}]
Let $A$ be a finitely generated minimal (co)chain $P$-algebra of finite type. Then $\Aaut(A)$ is an algebraic group that admits a normal unipotent algebraic subgroup $\mathscr K(A)$ where $\mathscr K(A)(K)$ is the group of automorphisms of $A\ot_k K$ homotopic to the identity for any field $K\supseteq k$. 
\end{lemma}
\begin{proof}
    The statement of \cite[Theorem 3.4]{BlockLazarev} is about cochain commutative algebras, but their proof holds word for word for (co)chain algebras over operads in vector spaces. Explicitly, $\mathscr K(A)$ is given by the exponential $\exp(B^0(\Der(A)))$ of the Lie algebra of zero boundaries in the cochain Lie algebra of derivations on $A$.
\end{proof}

\begin{dfn}
    Let $A$ be a finitely generated minimal (co)chain $P$-algebra of finite type. We define $\Aaut^h(A)$ to be the quotient 
    $\Aaut(A)/\mathscr K(A)$. 
\end{dfn}
\begin{rmk}
    Note that for any field $K\supseteq k $ we have that 
    $$\Aaut^h(A)(K) \cong(\Aaut(A)/\mathscr K(A))(K) \cong \Aaut(A)(K)/\mathscr K(A)(K) \cong\Aut^h(A\ot K)$$ 
    (the second isomorphism follows from Lemma \ref{wilkerson}).
\end{rmk}

\begin{dfn}
    Given an algebraic group $\mathscr G$, let $\mathscr G^\op$ be the algebraic group that has the same underlying variety as $\G$, but where the multiplication $\mu^\op\colon \G^\op\times\G^\op\to\G^\op$ is given by $\mu^\op(x,y) =\mu(y,x)$ where $\mu$ is the multiplication morphism of $\G$.
\end{dfn}

\begin{lemma}\begin{itemize}
    \item[(a)] $\G$ and $\G^\op$ are isomorphic algebraic groups.
    \item[(b)] Given a morphism of algebraic groups $f\colon \mathscr G\to \mathscr H$, then the same morphism of the underlying varieties
    $f_{\op}\colon \mathscr G^\op\to \mathscr H^\op$ is also a morphism of algebraic groups.
\end{itemize}
 
\end{lemma}
\begin{proof} (a)
    The inversion map $\mathrm{inv}\colon \mathscr G\to\mathscr G^{\mathrm{op}}$ gives the desired isomorphism.

(b) This is straightforward verification. 
\end{proof}

\begin{lemma}\label{lemma:ZariskiThm}
Let $\G$ and $\mathscr H$ be two algebraic groups over a field $k\supseteq \Q$, and let $\varphi\colon \G\to \mathscr H$ be a morphism of varieties that induces a group homomorphism (isomorphism)
  $$
  \varphi(\overline k)\colon \G(\overline k)\to \mathscr H(\overline k)
  $$
  on the $\overline k$-points. Then $\varphi$ is a morphism (isomorphism) of algebraic groups.
\end{lemma}
\begin{proof} 
    Since a morphism of varieties is completely determined by the induced map on $\overline k$-points, it follows that 
    $\varphi\circ \mu_\G = \mu_{\mathscr H}\circ (\varphi\times \varphi)$ if and only if they induce the same map on $\overline k$-points. Hence, $\varphi$ is a morphism of algebraic groups.

    Since algebraic groups over fields of characteristic zero are smooth \cite[Corollary 8.39]{milne}, they are normal, and since every component of an algebraic group is irreducible \cite[Corollary 1.35]{milne}, the isomorphism statement follows from \cite[Proposition 5.3]{zariski-main} (which is a consequence of Grothendieck’s form of Zariski's Main Theorem).
\end{proof}

\subsection{Operadic Koszul duality and homotopy of maps} We assume familiarity with the basic theory of algebraic operads, covered e.g. in \cite{LodayVallette}.  Readers not familiar with it, but familiar with rational homotopy theory could replace the bar-construction in below by the reduced Chevalley-Eilenberg coalgebra constructions and the cobar-constructions by Quillen's Lie algebra construction (both discussed e.g. in \cite[§ 22]{felixrht}), and might view the chain $P$-algebras as chain Lie algebras while the (co)chain $P^!$-(co)algebras as  non-unital (co)chain commutative (co)algebras.
 
\begin{conv}
    We will view dualization as a contravariant construction that takes chain complexes to cochain complexes and vice versa. In particular, the dual of a chain coalgebra is a cochain algebra. 

    We say that a chain $P$-coalgebra $C$ is connected (1-connected) if its dual, $C^\vee$ is connected (1-connected) (viewed as a cochain $P$-algebra). In particular, if $C$ is connected, then $C^0 = P(0)^\vee$ and if it is 1-connected, then $C$ is connected and $C^1=0$.
\end{conv}

We fix a Koszul operad $P$ and we let its Koszul dual operad be denoted by $P^!$.  Since we will consider both $P$ and $P^!$-algebras, we will denote a $P$-algebra by a letter in Fraktur font, e.g. $\mathfrak a$, while a  $P^!$-algebras is denoted by capital letter, e.g. $A$.   Associated to a Koszul operad there is an adjunction 
$$
B\colon \{\text{chain }P\text{-algebras}\} \rightleftarrows \{\text{conilpotent chain }P^!\text{-coalgebras}\}  :\Omega
$$
such that:
\begin{itemize}
    \item $B$ preserves quasi-isomorphisms and the  underlying graded $P^!$-coalgebra structure of $B(\mathfrak a)$ is given by the cofree graded $P^!$-coalgebra $P^!_c(s \mathfrak a) = \bigoplus P^!(i)\ot_{\Sigma_i} (s\mathfrak a)^{\ot i}$, where $(s\mathfrak a)^n \cong \mathfrak a^{n-1}$.
    \item $\Omega$ preserves quasi-isomorphisms between connected chain $P^!$-coalgebras  and the  underlying graded $P$-algebra structure of $\Omega(C)$ is given by the  free graded $P$-algebra $P(s^{-1} C) = \bigoplus P(i)\ot_{\Sigma_i} (s^{-1}C)^{\ot i}$, where $(s^{-1}C)^n \cong C^{n+1}$.
    \item The counit of the adjunction $\Omega B \mathfrak a\to \mathfrak a$ is a natural quasi-isomorphism for every chain $P$-algebra $\mathfrak a$.
    \item The unit of the adjunction $C\to B\Omega C$ is a natural quasi-isomorphism for every connected chain $P^!$-coalgebra $C$.
\end{itemize}
We refer the reader to \cite[§ 11.2-11.3]{LodayVallette} for proofs.
\begin{dfn}
    Let $$\mathcal C^*\colon\{\text{chain }P\text{-algebras}\} \to \{\text{cochain }P^!\text{-algebras}\} $$ denote the contravariant functor given by $\mathcal C^*(\a)= B(\a)^\vee$.

    The dual of a connected cochain $P^!$-algebra of finite type is a conilpotent $P^!$-coalgebra, and thus in the domain of $\Omega$. Let 
    $$
    \mathcal L\colon\{\text{connected cochain $P^!$-algebras of finite type}\} \to \{\text{chain }P\text{-algebras}\} 
    $$
    denote the contravariant functor given by $\mathcal L(A)= \Omega(A^\vee)$.
\end{dfn}

\begin{lemma}[\text{cf. \cite[Chapitre II, §6]{tanre}}]\label{lemma:preserve-htpy}
\begin{itemize}
    \item[(a)] Given a connected finite type chain $P$-algebra $\a$, there is a natural quasi-isomorphism $\tilde c\colon\mathcal L(\Cc(\a))\to\a$. 
    \item[(b)]  Given a connected finite type cochain $P^!$-algebra $A$, there is a natural quasi-isomorphism $\tilde u\colon\mathcal C^*(\mathcal L(A))\to A$.
    \item[(c)] If $f,g\colon \a\to \a'$ are homotopic, then $\Cc(f),\Cc(g)\colon\Cc(\a')\to\Cc(\a)$ are homotopic.
    \item[(d)] If $f,g\colon A\to A'$ are homotopic, $A$ and $A'$ are connected and quasi-free of finite type and $A$ is 1-connected, then $\mathcal L(f),\mathcal L(g)\colon \mathcal L(A')\to \mathcal L(A)$ are homotopic.
\end{itemize}
\end{lemma}
\begin{proof}
\textit{(a)} and \textit{(b)} are straightforward consequences of the unit and the counit of the bar-cobar adjunction being quasi-isomorphisms (under some connectivity assumptions).

For (c), we observe that $\Cc$ preserves quasi-isomorphisms (since $B$ does), and hence taking right homotopies to left homotopies. Since every algebra is fibrant (in the standard model category), the notion of right and left homotopy are equivalent.

For (d), we can not apply $\mathcal L$ to a right homotopy $A\to A'\ot\Lambda(t,dt)$, since the codomain is not of finite type. Instead, we need to consider left homotopies. If $A = (P^!(V),d)$, then a very good cylinder object for $A$ is given by
$(P^!(V\oplus \hat V\oplus s^{-1}\hat V),D)$ where $\hat V  \cong V$ and where $D(v)=d(v)$, $D(\hat v)= 0$ and $D(s^{-1}\hat v) = \hat v$. We observe that if $A$ is 1-connected then the specified cylinder object is connected, and thus we might apply $\mathcal L$ to the left homotopy in order to get a right homotopy.
\end{proof}

\begin{rmk}
    Note that $\Cc$ preserves homotopies in general, but $\mathcal L$ only preserves homotopies if the domain of the maps is 1-connected. This can be viewed as an algebraic counterpart to the topological fact that if two maps are pointed homotopic then they are  homotopic, but the reverse is only true in general if their codomain is simply connected.
\end{rmk}

\begin{lemma}\label{lemma:finite-coh-finite-gen}
    Let $A$ be a 1-connected quasi-free cochain $P^!$-algebra of finite type. Then cohomology  $H^*(A)$ is finite dimensional if and only if the minimal model for $\mathcal L(A)$ is finitely generated.
\end{lemma}

\begin{proof}
    This follows basically from the homotopy transfer theorem (\cite[§ 10.3]{LodayVallette}). A minimal $P^!_\infty$-model for $A$ is equivalent to a minimal coalgebra model for $BA$. If $A$ is 1-connected and of finite type, then $\Omega A^\vee$ is of finite type, and thus having a dual with an induced coalgebra structure, $\Omega(A^\vee)^\vee$, isomorphic to $BA$. Hence a minimal model $P(W)$ for $\Omega(A^\vee)$ is equivalent to a minimal coalgebra model $P_C(W^\vee)$ to $BA$, which in turn is equivalent to a minimal $P^!_\infty$-minimal model for $A$. By the homotopy transfer theorem, the underlying vector space of the minimal $P^!_\infty$-algebra model for $A$ is the cohomology $H^*(A)$, giving that $W^\vee = s^{-1}H^*(A)$. Hence $W$ is finite dimensional.
\end{proof}

\subsection{Main theorem}
\begin{thm}\label{main-thm}
    Let $A= (P^!(V),d)$ be a finite type 1-connected minimal cochain $P^!$-algebra whose cohomology $H^*(A)$ is concentrated in degrees $\leq n$ and let $\mathfrak a = P(W)$ be a minimal chain P-algebra model for $\mathcal L(A)$. Then there is an isomorphism of algebraic groups
    $$
    \Aaut^h(\mathfrak a)\cong \Aaut^h(A_n).
    $$
\end{thm}

\begin{proof}
    By the proof of Lemma \ref{lemma:finite-coh-finite-gen} we have that $W \cong (s^{-1} H^*(A))^\vee$. Hence $W = W^{\leq n-1}$. Consequently, we might identify $\Aaut(\a)$ with $\G_P^{n-1}(\a)\cong \G_P^n(\a)\cong \dots$ (see Remark \ref{rmk:aut-as-gn+t}). 

    Given a morphism $f\colon \a\to \a'$ of $P$-algebras, we have that $\Cc(f)\colon \Cc(\a')\cong P^!(s\a')\to \Cc(\a)\cong P^!(s\a)$ is given by $\Cc(f)(sx)=sf^\vee(x)$, for every $sx\in s\a'$ (this completely determines $\Cc(f)$).  
    
    We fix a basis $\{v_1,...,v_m\}$ for $\a^{\leq n}$, and we fix a basis for $C^*(\a)^{\leq n+1} = P^!(s\a)^{\leq n+1}$ consisting of elements of the form $\mu_t\ot sv_{i_1}\ot\cdots\ot sv_{i_t}\in P(t)\ot_{\Sigma_t}(s\a)^{\ot t}$.
    
    Hence there is a finite set of polynomials $\{p_{i,j}\}$ with variables being the entries of an $\dim(\a^{\leq n})\times \dim(\a^{\leq n})$-matrix  satisfying the following: If $M(\varphi,n)\in \mathfrak{gl}(\a^{\leq n})$ is the matrix representation of the restriction of an endomorphism $\varphi\colon\a\to\a$ to $\a^{\leq n}$, then the $(i,j)$'th entry in the matrix representation $M(\mathcal C^*(\varphi),n+1)\in \mathfrak{gl}(\mathcal C^*(\a)^{\leq n+1})$ for the restriction of $\mathcal C^*(\varphi)$ to $\mathcal C^*(\a)^{\leq n+1}$ is given by $p_{i,j}$ applied to the entries of $M(\varphi,n)$ (in particular, the map $\G^n_P(\a)\to \G^{n+1}_{P^!}(\Cc(\a))^{\op}$, defined by $M(\varphi,n)\mapsto M(\mathcal C^*(\varphi),n+1)$ is a morphism of algebraic groups).

    Since $\Cc(\a)$ and $A$ are quasi-isomorphic and both are cofibrant, there is some quasi-isomorphism $\eta\colon A\to \Cc(\a)$ and a homotopy inverse $\nu\colon \Cc(\a)\to A$. For every automorphism $\varphi\in \Aut(\a)$, we have that 
    $\nu\circ\Cc(\varphi)\circ\eta\colon A\to A$ is a quasi-isomorphism. Since $A$ is minimal $\nu\circ\Cc(\varphi)\circ\eta$ is an automorphism of $A$. 
    
    Now, we fix a basis for $A^{\leq n+1}$ that extends some basis for $A_n^{\leq n+1}\subset A^{\leq n+1}$.
    
    The restrictions  $\eta^{\leq n+1}\colon A^{\leq n+1}\to \Cc(\a)^{\leq n+1}$ and   $\nu^{\leq n+1}\colon\Cc(\a)^{\leq n+1}\to A^{\leq n+1}$ has matrix representations $M(\eta,n+1)$ and $M(\nu,n+1)$ respectively. Now the map
    $
    \G^n_P(\a)\to \G^{n+1}_{P^!}(A)^{\op}
    $
    given by
    $$
    M(\varphi,n)\mapsto  M(\nu\circ\Cc(\varphi)\circ\eta,n+1)=M(\nu,n+1)\cdot M(\mathcal C^*(\varphi),n+1)\cdot M(\eta,n+1)
    $$
    is a morphism of varieties (but not of algebraic groups, since group multiplications are not preserved). Since an automorphism of $A$ restricts to an automorphism of $A_n$ and since we chose a basis for $A^{\leq n+1}$ that extends a basis for  $A^{\leq n+1}_n$, the matrix representation\break $M(\mathrm{res}(\nu\circ\Cc(\varphi)\circ\eta, A_n),n+1)$ for the restriction of $\nu\circ\Cc(\varphi)\circ\eta$ to $A_n^{\leq n+1}$, might be identified with a submatrix of  $M(\nu\circ\Cc(\varphi)\circ\eta,n+1)$. Hence, again we have a morphism of varieties 
    $$
    \zeta\colon\Aaut(\a)\cong \G^n_P(\a)\to \G^{n+1}_{P^!}(A_n)^{\op} \cong \Aaut(A_n)^{\op}.
    $$

    Let $\pi\colon \Aaut(A_n)\to\Aaut^h(A_n)=\Aaut(A_n)/\mathscr K(A_n)$ denote the projection morphism. We claim that 
    $\pi_{\op}\circ\zeta\colon \Aaut(\a)\to\Aaut^h(A_n)^{\op}$
    is a morphism of algebraic groups.

    We have that $\pi_{\op}\circ\zeta$ is a morphism of varieties. Hence it is enough to prove that $\pi_{\op}\circ\zeta$ induces a group homomorphism on the $\overline k$-points by Lemma \ref{lemma:ZariskiThm}.

    In order to avoid cumbersome notations like $A\ot_k \overline k$ and $\a\ot_k\ok$, I will just assume that we are working over an algebraically closed field $\overline k$.

    We have that 
    $$\pi_\op\circ\zeta(\overline k)\colon \Aut(\a)\to \Aut^h(A_n)^{\op}$$
    $$\varphi\mapsto [\mathrm{res}(\nu\circ\Cc(\varphi)\circ  \eta,A_n)]$$ 
    factors through the map $$
    \rho\colon\Aut(\a)\to \Aut^h(A)^{\op}
    $$
    $$\varphi\mapsto  [\nu\circ\Cc(\varphi)\circ\eta] $$
    along the restriction induced map $\Aut^h(A)^\op\to\Aut^h(A_n)^\op$. By Theorem \ref{thm:iso-of-auto}, the restriction induced map $\Aut^h(A)^{\op}\to \Aut^h(A_n)^{\op}$ is a group isomorphism and hence it is enough to prove that $\rho\colon\Aut(\a)\to \Aut^h(A)^{\op}$ is a group homomorphism in order to deduce that $\pi_\op\circ\zeta(\overline k)$ is a group homomorphism.

    We have that 
    $$ \rho(\id) = [\nu\circ \eta] = [\id]$$
    and that 
$$    
     \rho(\varphi\circ \psi) = [\nu\circ\Cc(\varphi\circ\psi)\circ\eta]     
    = [\nu\circ\Cc(\psi)\circ\Cc(\varphi)\circ\eta] $$
    $$
   = [(\nu\circ\Cc(\psi)\circ\eta)\circ(\nu \circ\Cc(\varphi)\circ\eta)] = \rho(\varphi) * \rho(\psi)
   $$ 
    which proves that $\rho$ is a group homomorphism and thus $\pi_\op\circ \zeta$ is a morphism of algebraic groups.

    Since $\Cc$ preserves homotopies (Lemma \ref{lemma:preserve-htpy} \textit{(c)}), it follows that $\pi_\op\circ\zeta(\ok)$  restricts to the trivial morphism on $\mathscr K(\a)(\ok)\subseteq \Aut(\a)$. Hence $\pi_{\op}\circ\zeta$ factors through $\Aaut^h(\a)= \Aaut(\a)/\mathscr K(\a)$ and thus there is an induced morphism of algebraic groups
    $$
    F\colon \Aaut^h(\a) \to \Aaut^h(A_n)^{\op}.
    $$
    By Lemma \ref{lemma:ZariskiThm}, it is enough to show that $F$ induces an isomorphism on the group of $\ok$-points.
    Again, we have that $F(\ok)\colon\Aut^h(\a)\to\Aut^h(A_n)$ factors through a map $\chi\colon\Aut^h(\a)\to\Aut^h(A)$, $\chi([\varphi]) = [\nu\circ\Cc(\varphi)\circ \eta]$ along the restriction induced isomorphism $\Aut^h(A)\to \Aut^h(A_n)$. Hence it is enough to prove that $\chi$ is an isomorphism.

    \underline{\textit{$\chi$ is surjective:}} Since $\a$ and $\mathcal L(A)$ are quasi-isomorphic and cofibrant, there is some quasi-isomorphism $\kappa\colon\a\to\mathcal L(A)$ and a homotopy inverse $\tau\colon \mathcal L(A)\to \a$.
    
    Let $j\colon A\to A$ be the composition
    $$
    j\colon A\xrightarrow\eta \Cc(\a)\xrightarrow{\Cc(\tau)}\Cc(\mathcal L(A))\xrightarrow{\tilde u} A.
    $$
    Since $j$ is a composition of quasi-isomorphisms, and since $A$ is minimal, it follows that $j$ is an isomorphism.

    Let $\tilde u^{-1}\colon A\to \Cc(\mathcal L(A))$ be any homotopy inverse to the quasi-isomorphism $\tilde u\colon\Cc(\mathcal L(A))\to A$ and let $\ell\colon A\to A$ be the composition 
    
    $$
    \ell\colon A\xrightarrow{\tilde u^{-1}}\Cc(\mathcal L(A))\xrightarrow{\Cc(\kappa)}\Cc(\a)\xrightarrow{\nu}A
    $$
    which is a homotopy inverse to $j$ (but not necessarily the  strict inverse).
    
    For a homotopy class $[\Phi]\in \Aut^h(A)$, we pick some representative $\Phi\in \Aut(A)$, consider the automorphism $\varphi$ of $\a$ given by
    $$
    \varphi\colon\a\xrightarrow{\kappa}\mathcal L(A)\xrightarrow{\mathcal L(j^{-1}\Phi \ell^{-1})} \mathcal L(A)\xrightarrow{\tau} \a.
    $$
    We have that  $\chi([\varphi])= [\nu\circ \Cc(\varphi)\circ \eta]$.  We will show that $\nu\circ \Cc(\varphi)\circ \eta \simeq \Phi$, and thus  $\chi([\varphi])=[\Phi]$. 
    Consider the diagram
    $$
    \xymatrix@C=3.5em{
    A\ar[r]^-\eta\ar@{=}[d] & \Cc(\a)\ar[r]^-{\Cc(\tau)} & \Cc(\mathcal L(A))\ar[rr]^-{\Cc(\mathcal L(j^{-1}\Phi \ell^{-1}))}\ar[d]_{\tilde u} && \Cc(\mathcal L(A))\ar[r]^-{\Cc(\kappa)}\ar[d]_{\tilde u}& \Cc(\a)\ar[r]^-\nu & A \ar@{=}[d]
    \\
    A\ar[rr]_-j& & A\ar[rr]_-{j^{-1}\Phi \ell^{-1}} && A \ar[rr]_-\ell& & A. 
    }
    $$
    We have that the top row composes to $\nu\circ \Cc(\varphi)\circ \eta$ while the bottom row composes to $\Phi$. The diagram is homotopy commutative; the left square commutes strictly (by definition of $j$), the middle square commutes because of the naturality of $\tilde u$ and the right square commutes up to homotopy by definition of $\ell$. From this we conclude that $\nu\circ \Cc(\varphi)\circ \eta \simeq \Phi$, and thus  $\chi([\varphi])=[\Phi]$.

    \underline{\textit{$\chi$ is injective:}}
    Similar to above, we construct automorphisms $J,L\colon \a\to \a$, being homotopy-inverses to each other and making for every automorphism $\varphi\in\Aut(\a)$ the following diagram 
    $$
    \xymatrix@C=3.5em{
    \a\ar[r]^-\kappa\ar@{=}[d] & \mathcal L(A)\ar[r]^-{\mathcal L(\nu)} & \mathcal L(\Cc(\a))\ar[rr]^-{\mathcal L(\Cc(J^{-1}\varphi L^{-1}))}\ar[d]_{\tilde c} && \mathcal L(\Cc(\a))\ar[r]^-{\mathcal L(\eta)}\ar[d]_{\tilde c}& \mathcal L(A)\ar[r]^-\tau & \a \ar@{=}[d]
    \\
    \a\ar[rr]_-J& & \a\ar[rr]_-{J^{-1}\varphi L^{-1}} && \a \ar[rr]_-L& & \a 
    }    
    $$
    homotopy commutative.
    
    Assume now that $\chi([\varphi]) =[\nu\circ\Cc(\varphi)\circ\eta]= [\id]$. In particular, for any representative $\varphi$ of $[\varphi]$ we have that $\nu\circ\Cc(\varphi)\circ \eta\simeq \id$. Since $\nu$ and $\eta$ are homotopy inverses to each other, it also follows that $\Cc(\varphi)\simeq\id$. For such a $\varphi$, we have that the top row composes to an automorphism homotopic to the identity since $\mathcal L(\Cc(J^{-1}\varphi L^{-1}))=\mathcal L(\Cc(J^{-1}))\mathcal L(\Cc(\varphi))\mathcal L(\Cc(L^{-1})) \simeq \mathcal L(\Cc(J^{-1}))\mathcal L(\Cc(L^{-1}))\simeq \id$ (recall that $\mathcal C^*$ and $\mathcal L$ preserves  homotopic maps by Lemma \ref{lemma:preserve-htpy} under certain conditions that are fulfilled in this case), $\mathcal L(\nu)$ and $\mathcal L(\eta)$ are homotopy-inverse to each other, and $\kappa$ and $\tau$ are homotopy-inverse to each other. Hence, $\varphi$, which is the composition of the bottom row, is also homotopic to the identity. Thus $[\varphi]=[\id]$.
\end{proof}

\begin{exmp}
    We have that Theorem \ref{main-thm-rht} (in the introduction) is a special case of Theorem \ref{main-thm}. If $\L W$ is a 1-connected minimal chain Lie algebra model for a simply connected space $X$, then a minimal model for $\Cc(\L W)$ is the reduced minimal Sullivan model for $X$. We explain what we mean by reduced; if $\Lambda V$ is a minimal model for $X$, then $\overline{\Lambda V} = \Lambda V/k1$ is called the reduced minimal Sullivan model for $X$. However, since an algebra map between unital algebras has to send 1 to 1, it follows that $\map(\Lambda V,\Lambda V')=\map(\overline{\Lambda V},\overline{\Lambda V'})$ and $\Aut^h(\Lambda V)\cong \Aut^h(\overline{\Lambda V})$.
\end{exmp}

\begin{exmp} Let $X$ be finite simply connected $n$-dimensional CW-complex and let\break $(T_{\mathrm{cochain}}(V),d)$ be a minimal cochain associative algebra model for the singular cochain algebra $C^*(X;k)$ and $(T_{\mathrm{chain}}(W),\delta)$ be a minimal chain associative algebra model for the singular chains on the based  loops $C_*(\Omega X;k)$ with the Pontryagin product. The associative operad is Koszul dual to itself and the associated bar construction takes $C_*(\Omega X;k)$ to a coalgebra quasi-isomorphic to $C_*(X;k)$. By Theorem \ref{main-thm}, it follows that
$$
\Aut^h(T_{\mathrm{cochain}}(V),d)\cong \Aut^h(T_{\mathrm{cochain}}(V^{\leq n}),d) \cong \Aut^h(T_{\mathrm{chain}}(W),\delta)
$$
where the last isomorphism is induced by an isomorphism of algebraic groups.
\end{exmp}

\bibliographystyle{amsalpha}
\bibliography{references}
\noindent
\Addresses
\end{document}